\documentclass[11pt, notitlepage]{article}   % The document class has to be "report"

\usepackage{amsmath,amsthm,amsfonts}   % These are to write mathematics

\usepackage{amscd}
\usepackage{amsfonts}
\usepackage{amssymb}
\usepackage{color}      % Need the color package
\usepackage{epsfig}
\usepackage{graphicx}           % Graphics package

\theoremstyle{plain}
\newtheorem{Thm}{Theorem}[section]

\newtheorem{Prop}[Thm]{Proposition}
\newtheorem{Obs}[Thm]{Observation}

\theoremstyle{definition}
\newtheorem{Def}[Thm]{Definition}

\theoremstyle{remark}

\newcommand{\specsym}{\begin{picture}(8,8)
\scalebox{0.6} {\put(0,0){\framebox(10,10)}
\put(0,10){\line(1,-1){10}}}
\end{picture}}

\def\finf{\mathop{{\rm I}\kern -.27 em {\rm F}}\nolimits}

\begin{document}

\title{Iteration Index of a Zero Forcing Set\\ in a Graph}

\author{{\bf Kiran B. Chilakamarri$^1$}, {\bf Nathaniel Dean$^2$},\\ 
{\bf Cong X. Kang$^3$}, and {\bf Eunjeong Yi$^4$}\\
\small $^{1}$Texas Southern University, Houston, TX 77004, USA\\
$^1${\small\em chilakamarrikb@tsu.edu}\\
\small $^{2}$Texas State University, San Marcos, TX 78666, USA\\
$^2${\small\em nd17@txstate.edu}\\
\small $^{3,4}$Texas A\&M University at Galveston, Galveston, TX 77553, USA\\
$^3${\small\em kangc@tamug.edu}; $^4${\small\em yie@tamug.edu}}

\maketitle

\date{}

\begin{abstract}
Let each vertex of a graph $G=(V(G), E(G))$ be given one of two
colors, say, ``black" and ``white". Let $Z$ denote the (initial)
set of black vertices of $G$. The \emph{color-change rule}
converts the color of a vertex from white to black if the white
vertex is the only white neighbor of a black vertex. The set $Z$
is said to be \emph{a zero forcing set} of $G$ if all vertices of
$G$ will be turned black after finitely many applications of the
color-change rule. The \emph{zero forcing number} of $G$ is the
minimum of $|Z|$ over all zero forcing sets $Z\subseteq V(G)$.
Zero forcing parameters have been studied and applied to the
minimum rank problem for graphs in numerous articles. We define \emph{the iteration index of a zero forcing set of a graph $G$} to
be the number of (global) applications of the color-change rule
required to turn all vertices of $G$ black; this leads to a new graph invariant, the \emph{iteration index of $G$} -- it is the minimum of iteration 
indices of all minimum zero forcing sets of $G$. We present some basic properties of the iteration index and
discuss some preliminary results on certain graphs.
\end{abstract}

\noindent\small {\bf{Key Words:}} zero forcing set, zero forcing number, iteration index of a zero forcing set, Cartesian product of graphs, bouquet of 
circles\\

\small {\bf{2000 Mathematics Subject Classification:}} 05C50, 05C76, 05C38, 05C90\\

\section{Introduction}

The notion of a zero forcing set, as well  as the associated zero forcing number, of a simple graph was introduced in~\cite{AIM}
to bound the minimum rank for numerous families of graphs. Zero forcing parameters were further studied and applied to
the minimum rank problem in~\cite{preprint, UB, B2, Huang, Pathcover}. In this paper, we introduce and study the iteration index of a zero forcing set in 
a graph; as we'll see, this is a very natural graph parameter associated with a minimum zero forcing set of a graph.
After the requisite definitions and notations on the graphs to be considered, we'll give a brief review of the notions
and results associated with the zero forcing parameter.\\

In this paper, a graph $G = (V(G),E(G))$ has no isolated vertices and is finite, simple, and undirected.
The \emph{degree} of a vertex $v$ in $G$ is denoted by $\deg_G(v)$, and an \emph{end-vertex} is the vertex of degree one. The
minimum degree over all vertices of $G$ is denoted by $\delta(G)$. For
$S \subseteq V(G)$, we denote by $<\!S\!>$ the subgraph of $G$
induced by $S$, and we denote by $G-S$ the subgraph of $G$ induced
by $V(G) \setminus S$. For any vertex $v \in V(G)$, the \emph{open
neighborhood} of $v$ in $G$, denoted by $N_G(v)$, is the set of
all vertices adjacent to $v$ in $G$, and the \emph{closed
neighborhood} of $v$, denoted by $N_G[v]$, is the set $N_G(v) \cup
\{v\}$; we drop $G$ when ambiguity is not a concern. We denote by $P_n$, $C_n$, and $K_n$ the path,
the cycle, and the complete graph, respectively, on $n$ vertices. A
complete bipartite graph with partite sets having $p$ and $q$
vertices is denoted by $K_{p,q}$. The \textit{path cover number}
$P(G)$ of $G$ is the smallest positive integer $m$ such that there
are $m$ vertex-disjoint paths $P^1, P^2, \ldots, P^m$ in
$G$ that cover all vertices of $G$. The \textit{Cartesian product}
of two graphs $G$ and $H$, denoted by $G \square H$, is the graph
with the vertex set $V(G) \times V(H)$ such that $(u,v)$ is
adjacent to $(u', v')$ if and only if
(1) $u=u'$ and $vv' \in E(H)$ or (2) $v=v'$ and $uu' \in E(G)$. For other graph theory terminology, we refer to \cite{CZ}.\\

Let each vertex of a graph be given either the color black or the color white. Denote by $Z$ the initial set of black
vertices. The ``color-change rule" changes the color of a vertex $w$ from white to black if
the white vertex is the only white neighbor of a black vertex $u$; in this case, we may say that
\emph{$u$ forces $w$} and write $u \rightarrow w$. Of course, there may be more
than one black vertex capable of forcing $w$, but we associate only one forcing vertex to $w$ at a time. Applying the color-change rule to all vertices
of $Z$, we obtain an \textit{updated} set of black vertices $Z_1\supseteq Z$. Clearly, not all vertices in $Z$ need to be forcing
vertices, and if a vertex $u$ in $Z$ forces $w$, then $u$ becomes
inactive -- i.e., unable to force thereafter. The vertex $w$
replaces $u$ as a potential forcing vertex in $Z_1$; thus, $Z_1$ has at most $|Z|$ many
potentially forcing vertices. Applying the color change rule to $Z_1$
results in another updated set $Z_2\supseteq Z_1$ of black vertices. Continuing this process until no more color change is possible, we obtain a nested 
sequence of sets $Z=Z_0 \subseteq Z_1  \subseteq \ldots \subseteq Z_n$. The initial set $Z$ is said to be \emph{a zero forcing set} if $Z_n=V(G)$. A
``chronological list of forces" is a record of the forcing actions in the
order in which they are performed. Given any chronological list of
forces, a ``forcing chain" is a sequence $u_1, u_2, \ldots, u_t$
such that $u_i \rightarrow u_{i+1}$ for $i=1,2, \ldots, t-1$.
In consideration of all lists of forces leading from vertices in $Z=Z_0$ to all vertices in $Z_i-Z_{i-1}$, we see that $|Z_i-Z_{i-1}|\leq |Z|$ for each 
$i\in\{1,2,\ldots,n\}$.
The zero forcing number of $G$, denoted by $Z(G)$, is the minimum of $|Z|$ over all zero forcing sets $Z \subseteq V(G)$.\\

A ``maximal forcing chain" is a forcing chain that is not a
subsequence of another forcing chain. If $Z$ is a zero forcing
set, then a ``reversal" of $Z$ is the set of last vertices of
maximal forcing chains of a chronological list of forces. The following are some of the known properties of zero forcing parameters:

\begin{itemize}
\item[$\diamond$] \cite{UB} For any graph $G$, $\delta(G) \le
Z(G)$.
\item[$\diamond$] \cite{preprint} If $Z$ is a zero forcing set of
a graph $G$, then any reversal of $Z$ is also a zero forcing set
of $G$.
\item[$\diamond$] \cite{preprint} If a graph $G$ has a unique zero
forcing set, then $G$ has no edges; i.e., $G$ consists of isolated
vertices.
\item[$\diamond$] \cite{preprint} For any graph $G$, $P(G) \leq
Z(G)$.
\item[$\diamond$] \cite{AIM} For any tree $T$, $P(T)=Z(T)$.
\end{itemize}

For two graphs $G$ and $H$ such that $H \subseteq G$, one cannot
exactly determine $Z(G)$ from $Z(H)$ -- or vice versa, but the following holds.

\begin{Thm} \cite{B2}
Let $G$ be any graph. Then
\begin{itemize}
\item[(i)] For $v \in V(G)$, $Z(G)-1 \le Z(G-\{v\}) \le Z(G)+1$.
\item[(ii)] For $e \in E(G)$, $Z(G)-1 \le Z(G-e) \le Z(G)+1$.
\end{itemize}
\end{Thm}

\section{Iteration index of a graph}

To facilitate the precise definition of the iteration index $I(G)$ of a graph $G$, we shall first more precisely
(and concisely) define zero forcing parameters in terms of a discrete dynamical system associated with $G$,
which we'll call the zero forcing system.\\

\begin{Def}
Given any graph $G$, the zero forcing system induced by a vertex set $S
\subseteq V(G)$ is the following recursively defined sequence of
functions $\chi_S^i: V(G) \rightarrow \{0,1\}$ such that
\begin{equation*}
\chi_S^0(v)=\left\{
\begin{array}{ll}
0 & \mbox{ if } v \in S\\
1 & \mbox{ if } v \in V(G)\setminus S;
\end{array} \right.
\end{equation*}
let $\chi_S^i$ be defined for $i \ge 0$, then
\begin{equation*}
\chi_S^{i+1}(v)=\left\{
\begin{array}{ll}
0 & \mbox{ if } \chi_S^i(v)=0\\
0 & \mbox{ if } \chi_S^i(v)=1, \exists u \in N(v) \mbox{ such that } \forall w \in N[u] \mbox{ with } w \neq v, \chi_S^i(w)=0\\
1 & \mbox{ otherwise }.
\end{array} \right.
\end{equation*}
\end{Def}

\begin{Def}
A vertex set $Z \subseteq V(G)$ is a \textit{zero forcing
set} if there exists some $n \ge 0$ such that $\chi_Z^n(v)=0$,
$\forall v \in V(G)$. For $i\ge 1$, we define the $i$-th
\textit{derived set} of $Z$, denoted by $D_Z^i$, as
$D_Z^{i}=\{v \in V(G) : \chi_Z^i(v)=0 \mbox{ and }
\chi_Z^{i-1}(v) =1\}$.
\end{Def}

\begin{Def}
The \textit{zero forcing number} $Z(G)$ is the minimum of $|Z|$ over all zero forcing sets $Z\subseteq V(G)$.
And a zero forcing set of cardinality $Z(G)$ will be called a $Z(G)$-set.
\end{Def}

Next, we define the iteration index of a graph, on which the present paper is focused.\\

\begin{Def}
For any zero forcing set $Z$ of $G$, the \textit{iteration index} $I_Z(G)$
of $Z$ is the minimum $n \ge 0$ such that $\chi_Z^n(v)=0$ for any
$v \in V(G)$. And the iteration index of $G$ is $I(G)= \min \{|I_Z(G)|: Z \mbox{ is a $Z(G)$-set}\}$.
\end{Def}

Note that $Z_i$ for $i\geq 1$ as defined earlier is
${(\chi_Z^i)}^{-1}(0)=Z\cup (\cup_{j=1}^{i}D_Z^{j})$. And thus
$I_Z(G)$ is the length $n$ of the strictly increasing sequence of
sets $Z=Z_0\subset \ldots Z_{n-1}\subset Z_n=V(G)$, where
$Z_{n-1}\neq V(G)$. In prose, the iteration index of a zero
forcing set $Z$ of $G$ is simply the number of global (taking all
black vertices at each step) applications of the color-change rule
required to effect all vertices of $G$ black, starting with $Z$.
The minimum ($I(G)$) among such values for all $Z(G)$-sets is then
an invariant of $G$ which is intrinsically interesting. From the
``real world" modeling (or discrete dynamical system) perspective,
here is a possible scenario: There are initially $|Z|$ persons
carrying a certain condition or trait (anything from a virus to a
genetic mutation) in a population of $|V(G)|$ people, where the
edges $E(G)$ characterize, say, inter-personal relation of a
certain type. If $Z$ is capable of passing the condition to the
entire population (i.e., ``zero forcing"), then $I_Z(G)$ may
represent the number of units of time (anything from days to millennia)
necessary for the entire population to acquire the condition or trait. \\

Now, we consider an algebraic interpretation of $I_Z(G)$; it is
related to Proposition 2.3 of~\cite{AIM} which states that: ``Let
$Z$ be a zero forcing set of $G=(V,E)$ and $A\in
\mathcal{S}(F,G)$. [Here, $A=(A_{ij})$ is a symmetric matrix where the diagonal entries are arbitrary elements of a field $F$ and, for $i \neq j$,
$A_{ij}\neq 0$ exactly when $ij$ is an edge of the graph $G$.] If $\mathbf{x}\in ker(A)$ and
$supp(\mathbf{x})\cap Z=\emptyset$, then $\mathbf{x}=\mathbf{0}$"
Rephrasing slightly, this says that $x_i=0$ (where
$\mathbf{x}=(x_i)$) for each $i\in Z$ and $A\mathbf{x}=\mathbf{0}$
together imply that $x_j=0$ for $j\notin Z$ as well. The length of
the longest forcing chain of $Z$" ($LLFC(Z)$) appears to be the
number of steps needed to reach $\mathbf{x}=\mathbf{0}$ by solving
the linear system $A\mathbf{x}=\mathbf{0}$ through \emph{naive
substitution}, starting with the data $x_i=0$ for each $i\in Z$
--- as Example 1 will show. It's clear that $I_Z(G)$
is an upper bound for the $LLFC(Z)$. It will be shown in
Example 2 that $I_Z(G)$ may be strictly greater than
$LLFC(Z)$. However, $I_Z(G)$ has the advantage of being
canonically defined, in contrast to the notion of the forcing
chain: After fixing a zero forcing set $Z$, an arbitrary choice
must be made when there are two or more forcing vertices at any
given step. Thus, there may be multiple reversals of $Z$; a
reversal of a reversal of $Z$ is not necessarily $Z$ --- to
name two of the side effects of the non-canonical nature of the forcing chain.\\

\textbf{Example 1.} Let $G=C_3 \square K_2$, with vertices
labeled as in Figure~\ref{Z}.
\begin{figure}[htbp]
\begin{center}
\scalebox{0.45}{\input{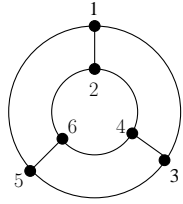}} \caption{$C_3 \square
K_2$}\label{Z}
\end{center}
\end{figure}
Notice $Z(G) \ge \delta(G)=3$. Since $\{2,4,6\}$ is a zero forcing
set, we have $Z(G)=3$. Since $Z(G)<|V(G)|$, we have $I(G)\ge 1$.
With $Z=\{2,4,6\}$, we get $D_Z^1=\{1,3,5\}$, and so $I_Z(G)=I(G)=1$.

\vspace{.1in}

Now, $Z'=\{3,4,6\}$ is also a $Z(G)$-set. Note that
$D_{Z'}^1=\{2\}$ and $D_{Z'}^2=\{1,5\}$. So
$I_{Z'}(G)=2$. Thus, we see that $I_Z(G)$ is not constant as $Z$ varies over $Z(G)$-sets.\\

Let $A$ be the generic symmetric matrix associated with the graph
in Figure~\ref{Z}; let $\mathbf{v}$ and $\mathbf{w}$ be supported
outside of $Z$ and $Z'$, respectively. Thus
\[A=\left(\begin{array}{cccccc}
a_1 & c_1 & c_2 & 0 &  c_3 & 0\\
c_1 & a_2 & 0 & c_4 & 0 & c_5\\
c_2 & 0 & a_3 & c_6 & c_7 & 0\\
0 & c_4 & c_6 & a_4 & 0 & c_8\\
c_3 & 0 & c_7 & 0 & a_5& c_9\\
0 & c_5 & 0 & c_8 & c_9 & a_6
\end{array} \right),
\mathbf{v}=\left(\begin{array}{c}
v_1\\
0\\
v_3\\
0\\
v_5\\
0
\end{array} \right), \ \ and\ \
\mathbf{w}=\left(\begin{array}{c}
w_1\\
w_2\\
0\\
0\\
w_5\\
0
\end{array} \right)
\]
where $a_i$'s are arbitrary and $c_j$'s are each non-zero real
numbers. First suppose $A\mathbf{v}=\mathbf{0}$. Then we have the
following system of linear equations:

\[\left\{ \begin{array}{ll}
a_1 \cdot v_1+c_2 \cdot v_3+c_3 \cdot v_5=0 & (1)\\
c_1 \cdot v_1=0 & (2)\\
c_2 \cdot v_1+ a_3 \cdot v_3+c_7 \cdot v_5=0 & (3)\\
c_6 \cdot v_3=0 & (4)\\
c_3 \cdot v_1+ c_7 \cdot v_3+ a_5 \cdot v_5=0 & (5)\\
c_9 \cdot v_5=0 . & (6)
\end{array}
\right. \]

From the second, fourth, and sixth equations above, we
get $v_1=v_3=v_5=0$. Here, we reach $\mathbf{v}=\mathbf{0}$ in one
step, corresponding to $LLFC(Z)=I_Z(G)=1$.\\

Next suppose $A\mathbf{w}=\mathbf{0}$. Then we have the following
system of linear equations:

\[\left\{ \begin{array}{ll}
a_1 \cdot w_1+c_1 \cdot w_2+c_3 \cdot w_5=0 & (1)\\
c_1 \cdot w_1+a_2 \cdot w_2=0 & (2)\\
c_2 \cdot w_1 + c_7 \cdot w_5=0 & (3)\\
c_4 \cdot w_2=0 & (4)\\
c_3 \cdot w_1+ a_5 \cdot w_5=0 & (5)\\
c_5 \cdot w_2+ c_9 \cdot w_5=0 . & (6)
\end{array}
\right. \]

From the fourth equation above, we get $w_2=0$, which corresponds
to $D_{Z'}^1$. By applying $w_2=0$ to the system of linear
equations above, we get

\[\left\{ \begin{array}{ll}
a_1 \cdot w_1+c_3 \cdot w_5=0 & (1)\\
c_1 \cdot w_1=0 & (2)\\
c_2 \cdot w_1 + c_7 \cdot w_5=0 & (3)\\
c_3 \cdot w_1+ a_5 \cdot w_5=0 & (4)\\
c_9 \cdot w_5=0. & (5)
\end{array}
\right. \]

The second and fifth equations yield $w_1=0=w_5$, which corresponds to $D_{Z'}^2$.
Here, we reach $\mathbf{w}=\mathbf{0}$ in two steps, corresponding to $LLFC(Z')=I_{Z'}(G)=2$.\\

\textbf{Example 2.} As discussed in the introduction, the notion
of ``a forcing chain" has been introduced and made use of
in~\cite{AIM} and elsewhere. However, the length of the longest
forcing chain of a graph can be strictly less than its iteration
index. For example, let $T$ be the tree in Figure \ref{remark}.
Since $P(T)=3$, $Z(T)=3$. One can readily check that there are ten
$Z(G)$-sets for $T$: $\{1,2,4\}$, $\{1,2,5\}$, $\{1,3,4\}$,
$\{1,3,5\}$, $\{2,4,5\}$, $\{2,4,9\}$, $\{2,5,9\}$, $\{3,4,5\}$,
$\{3,4,9\}$, and $\{3,5,9\}$. Further, one can check for $Z$ (any
of the ten sets) that the length of the longest forcing chain is
two, while $I_Z(T)=I(T)=3$. Let $A$ denote the generic symmetric
matrix associated with $T$, and let $\mathbf{v}\in ker(A)$ be a
vector supported outside of $Z$. Solving $A\mathbf{v}=\mathbf{0}$
through naive substitution starting with the data $v_i=0$ for
each $i\in Z$ as we did in the previous example, one sees that
the number of steps needed to reach $\mathbf{v}=0$ is $LLFC(Z)$ rather than $I(Z)$.\\

\begin{figure}[htbp]
\begin{center}
\scalebox{0.45}{\input{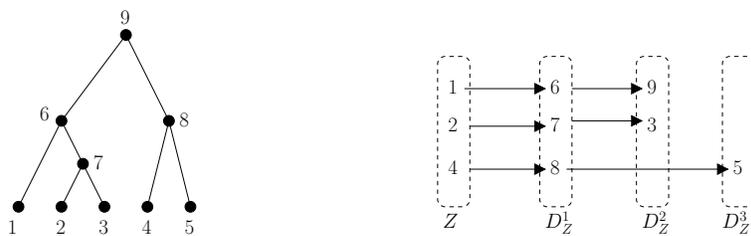}} \caption{An exmple
showing that $LLFC(Z)=2 < I_Z(G)=I(G)=3$}\label{remark}
\end{center}
\end{figure}

\begin{Thm}\label{bounds}
For any graph $G$ that is not edgeless, $$\max\left\{\frac{|V(G)|}{Z(G)}-1, 1\right\}\le
I(G) \le |V(G)|-Z(G).$$
\end{Thm}

\begin{proof}
Since $G$ has an edge, the cardinality of a minimum zero forcing set $Z_0$ is less than $|V(G)|$; thus, $I(G)\geq 1$. Application of
color-change rule results in a chain of sets $Z_0 \subset Z_1
\subset \cdots \subset Z_k=V(G)$, where $k=I(G)$. Note that
$|Z_0|=Z(G)$. Since $V(G)$ is the disjoint union of $Z_0$ and $Z_i-Z_{i-1}$ for $1 \le i \le k$,
we have $$|V(G)|=|Z_0|+\sum_{i=1}^k|Z_i-Z_{i-1}| \le
|Z_0|+k|Z_0|=(k+1)|Z_0|.$$ So $\frac{|V(G)|}{|Z_0|}-1 \le k
=I(G)$. The inequality $I(G) \le |V(G)|-|Z_0|$ is trivial since
$Z_i-Z_{i-1} \neq \emptyset$ for $1 \le i \le k$. This completes
the proof.
\end{proof}

\begin{figure}[htbp]
\begin{center}
\scalebox{0.45}{\input{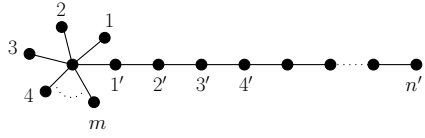}} \caption{A graph $G$
satisfying $I(G)=|V(G)|-Z(G)$}\label{upper bound}
\end{center}
\end{figure}

\textbf{Remark 1.} The bounds in Theorem~\ref{bounds} are best possible. We illustrate the sharpness of the lower bound with two examples. First notice 
$I(K_n)=1$ (see (i) of Observation \ref{observation}). Second, let $G$ be $P_s \square P_t$: for $s \ge t \ge 2$, one can readily check that
$|V(G)|=st$, $Z(G)=t$, and $I(G)=s-1=\frac{|V(G)|}{Z(G)}-1$. For the sharpness of the upper bound, $K_n$ again serves as an example. As a less trivial 
example, let $G$ be the graph obtained by joining the center of a star $K_{1,m}$ to an end-vertex of the
path $P_n$ (see Figure \ref{upper bound}): one can readily check
that $|V(G)|=m+n+1$, $Z(G)=m$, and $I(G)=n+1=|V(G)|-Z(G)$.\\

In \cite{AIM}, the zero forcing number for a cycle, a path, a
complete graph, and a complete bipartite graph (respectively) was obtained; i.e.,
(i) $Z(P_n)=1$ for $n \ge 2$; (ii) $Z(C_n)=2$ for $n \ge 3$, (iii)
$Z(K_n)=n-1$ for $n \ge 2$, and (iv) $Z(K_{p,q})=p+q-2$, where $p,
q \ge 2$.\\

\begin{Obs} \hfill \label{observation}
\begin{itemize}
\item[(i)] $I(K_n)=1$ for $n \ge 2$, since $Z(K_n)=n-1$.
\item[(ii)] $I(P_n)=n-1$ for $n \ge 2$, since $Z(P_n)=1$ and only an end-vertex is a minimum zero forcing set.
\item[(iii)] $I(C_n)=\lceil\frac{n-2}{2}\rceil$ for $n \ge 3$, since $Z(C_n)=2$ and only an adjacent pair of vertices is a minimum zero forcing set.
\item[(iv)] $I(K_{1,q})=2$ for $q \ge 2$, since $Z(K_{1,q})=q-1$ and any minimum zero forcing set must omit the central vertex along with an
    end-vertex.
\item[(v)] $I(K_{p,q})=1$ for $p, q, \ge 2$, since $Z(K_{p,q})=p+q-2$ and one may choose the minimum zero forcing set that
omits a vertex from each partite set.
\end{itemize}
\end{Obs}

\section{Zero forcing number and Iteration index of the Cartesian
product of some graphs}

Consider $G \square H$ with $|G|=s$ and $|H|=t$. Let the $t$
copies of $G$ to be $G^{(1)}, G^{(2)}, \ldots, G^{(t)}$ from the
left to the right and let the $s$ copies of $H$ to be $H^{(1)},
H^{(2)}, \ldots, H^{(s)}$ from the top to the bottom. The vertex labeled $(x,y)$ in $G \square H$ is the result of the
intersection of $G^{(y)}$ and $H^{(x)}$. See Figure \ref{grid} for
$G=P_s$ and $H=P_t$. The Cartesian product $P_s \square P_t$ is
also called a \textit{grid graph}.\\

\begin{figure}[htbp]
\begin{center}
\scalebox{0.45}{\input{grid1.pstex_t}} \caption{The grid graph
$P_s \square P_t$}\label{grid}
\end{center}
\end{figure}

In \cite{AIM}, it is shown that $Z(P_s \square P_t)= \min \{s,t\}$
for $s,t \ge 2$, $Z(C_s \square P_t)= \min\{s,2t\}$ for $s \ge 3,
t \ge 2$, $Z(K_s \square P_t)=s$ for $s,t \ge 2$, $Z(K_s \square
K_t)=st-s-t+2$ for $s,t \ge 2$, and $Z(C_s \square K_t)=2t$ for $s
\ge 4$.\\

\begin{Thm} \label{T1} \hfill
\begin{itemize}
\item[(i)] For $t \ge s \ge 2$, $I(P_s \square P_t)=t-1$.
\item[(ii)] For $s,t \ge 2$, $I(K_s \square P_t)=t-1$.
\item[(iii)] For $s \ge 3$ and $t \ge 2$,
$I(C_s \square P_t)=\left\{
\begin{array}{ll}
\lceil\frac{s-2}{2}\rceil & \mbox{ if } s \ge 2t\\
t-1 & \mbox{ if } s<2t .
\end{array} \right.$
\item[(iv)] For $s \ge 4$ and $t \ge 2$, $I(C_s \square K_t)=\lceil\frac{s-2}{2}\rceil$.
\end{itemize}
\end{Thm}

\begin{proof}
$(i)$ Since $Z(P_s \square P_t)=s$, $I(P_s \square P_t) \ge t-1$ by
Theorem \ref{bounds}. If we take $Z_0=\{(1,1), (2,1), \ldots, (s,1)\}$ as a $Z(P_s \square P_t)$-set (see Figure~\ref{grid}), then, for each $1 \le i \le 
s$, $(i,j) \rightarrow (i, j+1)$ if $1 \le j \le t-1$.
Hence $Z_{t-1}=V(P_s \square
P_t)$, so $I(P_s \square P_t) \le t-1$.\\

$(ii)$ Since $Z(K_s \square P_t)=s$, $I(K_s \square P_t) \ge t-1$ by
Theorem \ref{bounds}.
If we take $Z_0=\{(1,1), (2,1), \ldots, (s,
1)\}$ as a $Z(K_s \square P_t)$-set (see Figure \ref{KP} for $K_5 \square P_3$), then,
for each $1 \le i \le s$, $(i,j) \rightarrow (i, j+1)$ if $1 \le j \le t-1$.
Hence $Z_{t-1}=V(K_s \square P_t)$, so
$I(K_s \square P_t) \le t-1$.\\

$(iii)$ We consider two cases.

\vspace{.1in}

\textit{Case 1. $s \ge 2t$:} Then $Z(C_s \square P_t)=2t$, and $I(C_s
\square P_t) \ge \lceil\frac{s-2}{2}\rceil$ by Theorem
\ref{bounds}. By taking $Z_0=\{(1,1), (1,2), \ldots, (1,t)\} \cup
\{(s,1), (s,2), \ldots, (s,t)\}$ as a $Z(C_s \square P_t)$-set, we have, for each $1 \le j \le t$, that
$(i, j) \rightarrow (i+1, j)$ for $1 \le i \le \lceil\frac{s}{2}\rceil-1$ and $(i,j) \rightarrow (i-1,j)$ for $\lfloor\frac{s}{2}\rfloor +2 \le i \le s$.
Hence $Z_{\lceil\frac{s-2}{2}\rceil}=V(C_s \square P_t)$, so $I(C_s
\square P_t) \le \lceil\frac{s-2}{2}\rceil$.
Thus $I(C_s \square P_t)= \lceil\frac{s-2}{2}\rceil$.

\vspace{.1in}

\textit{Case 2. $s < 2t$:} Then $Z(C_s \square P_t)=s$, and $I(C_s
\square P_t) \ge t-1$ by Theorem \ref{bounds}. If we take
$Z_0=\{(1,1), (2,1), \ldots, (s,1)\}$ as a $Z(C_s \square
P_t)$-set, then, for each $1 \le i \le s$, $(i,j) \rightarrow (i, j+1)$ if $1 \le j \le t-1$.
Hence $Z_{t-1}=V(C_s \square
P_t)$, so $I(C_s \square P_t) \le t-1$. Thus $I(C_s \square
P_t)=t-1$.\\

$(iv)$ Since $Z(C_s \square K_t)=2t$, $I(C_s
\square K_t) \ge \lceil\frac{s-2}{2}\rceil$ by Theorem
\ref{bounds}. By taking $Z_0=\{(1,1), (1,2), \ldots, (1,t)\}$
$\cup \{(s,1), (s,2), \ldots, (s,t)\}$ as a $Z(C_s \square K_t)$-set, we have, for each $1 \le j \le t$, that
$(i, j) \rightarrow (i+1, j)$ for $1 \le i \le \lceil\frac{s}{2}\rceil-1$ and $(i,j) \rightarrow (i-1,j)$ for $\lfloor\frac{s}{2}\rfloor +2 \le i \le s$.
Hence $Z_{\lceil\frac{s-2}{2}\rceil}=V(C_s \square K_t)$, so $I(C_s
\square K_t) \le \lceil\frac{s-2}{2}\rceil$.
\end{proof}

\begin{figure}[htbp]
\begin{center}
\scalebox{0.45}{\input{KP.pstex_t}} \caption{The Cartesian product
$K_5 \square P_3$}\label{KP}
\end{center}
\end{figure}

Next we consider $K_s \square K_t$ (see Figure \ref{KK} for $K_5
\square K_4$). By (i) and (ii) of Theorem \ref{T1}, we have $I(K_2 \square K_2)=I(P_2 \square P_2)=1$ and $I(K_3 \square K_2)=I(K_3 \square P_2)=1$.\\

\begin{figure}[htbp]
\begin{center}
\scalebox{0.45}{\input{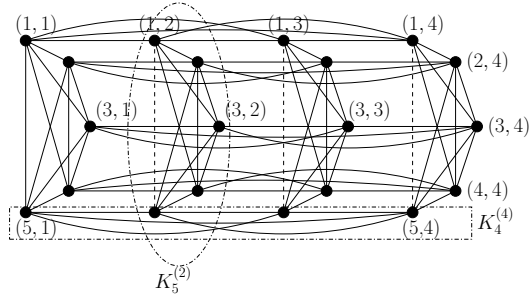}} \caption{The Cartesian product
$K_5 \square K_4$}\label{KK}
\end{center}
\end{figure}

\begin{Thm}\label{theorem on KK}
For $s,t \ge 3$, $I(K_s \square K_t)=2$.
\end{Thm}

\begin{proof}
Since $Z(K_s \square K_t)=st-s-t+2< |V(K_s \square K_t)|$, $I(K_s
\square K_t) \ge 1$. Assume, for the sake of contradiction, that
$I(K_s \square K_t)=1$. Take any $Z(K_s \square K_t)$-set $Z$ with
$I_Z(K_s \square K_t)=1$ and let $w_1, w_2, \ldots, w_x$ be an ordered listing of
all the vertices not in $Z$; i.e, the ``white vertices". Without loss of generality (WLOG), let
$t\geq s\geq 3$. Let $w_1$ be located in the $i$-th row \textbf{and} $j$-th
column (see Figure~\ref{KK}). Since $I_Z(K_s \square K_t)=1$, there must exist a
``black" vertex $b\in Z$ located in the $i$-th row \textbf{or} $j$-th column
such that each $v\in N[b] \setminus \{w_1\}$ is in $Z$, and $b$ forces $w_1$. Thus, $w_1$ implies the existence
of $s+t-2$ black vertices. Likewise, each $w_i$ implies the existence of (not counting
overlaps) $s+t-2$ black vertices -- namely a ``black row" and a
``black column", disregarding $w_i$ itself. Having considered all
the black rows and black columns (disregarding the $w_i$'s)
corresponding to vertices $w_1$ through $w_q$ for $1\leq q <x$, consider
$w_{q+1}$. Notice that either the black row or the black column
(again, disregarding $w_{q+1}$ itself) corresponding to $w_{q+1}$ must
be ``new", since either the corresponding row or corresponding column contains
$w_{q+1}$: this means that $w_{q+1}$ implies the existence of at least $s-1$ new
black vertices. We thus have the inequality $s+t-2+(x-1)(s-1)\leq
st-s-t+2$, which easily implies that $x<t$, contradicting the fact
$x=t+s-2$ and the hypothesis $t\geq s\geq 3$. Hence $I(K_s
\square K_t) \geq 2$. \\

On the other hand, if we take $Z_0= (\cup_{j=1}^{t-1}\{(2,j),(3,j), \ldots, (s,j)\}) \cup \{(1,1)\}$ as a $Z(K_s \square
K_t)$-set with $|Z_0|=st-s-t+2$, then, for each $2 \le i \le s$, $(i,1) \rightarrow (i,t)$; so $Z_1=Z_0 \cup \{(2,t), (3,t)
, \ldots, (s,t)\}=V(K_s \square K_t) \setminus \{(1,2), (1,3), \ldots, (1,t)\}$. Next, for each $2 \le j \le t$, $(s,j) \rightarrow (1,j)$, and thus $Z_2=V(K_s \square K_t)$. Therefore, $I(K_s \square K_t) \le 2$.
\end{proof}

\textbf{Remark 2.} Noticing $Z(C_3 \square K_t)=Z(K_3
\square K_t)=2t-1$, we have $I(C_3 \square K_2)=I(C_3 \square P_2)=1$ by (iii) of Theorem \ref{T1}, and $I(C_3 \square
K_t)=2$ for $t \ge 3$ by Theorem \ref{theorem on KK}.\\

We recall that 

\begin{Prop} (\cite{AIM}, Prop. 2.5) \label{prop 2.5 of [1]} 
For any graphs $G$ and $H$, $Z(G \square H) \le \min\{Z(G)|H|, Z(H)|G)|\}$.
\end{Prop}

\begin{Prop}
Let $t > s \ge 3$. Then\begin{equation*} \left\{
\begin{array}{lll}
Z(C_s \square C_s) \le 2s-1 & \mbox{ if $s$ is odd } & (1)\\
Z(C_s \square C_s) \le 2s & \mbox{ if $s$ is even } & (2)\\
Z(C_s \square C_t) \le 2s & & (3) .
\end{array} \right.
\end{equation*}
\end{Prop}

\begin{proof}
Since $Z(C_m)=2$ for any $m \ge 3$, parts (2) and (3) of the
conclusion follow immediately from Proposition \ref{prop 2.5 of [1]}; so we
only need to show part (1) of the conclusion. It's obvious that
$2s$ many black vertices on two adjacent cycles ($C_s$) form a
zero forcing set. It thus suffices to show that starting with
$2s-1$ black vertices on two adjacent cycles, after finitely many
applications of the color-change rule, one obtains two adjacent
cycles as a subset of the set of black vertices. This can be seen
as follows: Label the $s^2$ vertices on $C_s\square C_s$ by
$(i,j)$, where $1\leq i,\,j\leq s$. Take as the initial set of
black vertices $\{(i,j): 1\leq i\leq 2 \mbox{ and } 1\leq j\leq
s\} \setminus \{(1,\frac{s+1}{2})\}$ (recall that $s$ is odd). One can readily check (see Figure \ref{5,5}) that after $\frac{s-1}{2}$ applications of the color change
rule, the two adjacent cycles $\{(i,j): 1\leq i\leq s \mbox{ and }
j\in\{1,s\}\}$ will consist of only black vertices.
\end{proof}

\begin{figure}[htbp]
\begin{center}
\scalebox{0.45}{\input{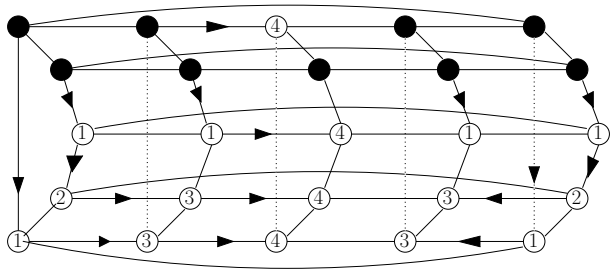}} \caption{The set of
black vertices is a zero forcing set $Z_0$ of $C_5 \square C_5$, the
number $m$ in each vertex indicates that the vertex is in $Z_m$, and the arrows indicate possible forcing chains corresponding to $Z_0$.
}\label{5,5}
\end{center}
\end{figure}

\textbf{Remark 3.} Note that $Z(C_3 \square C_3)=Z(K_3 \square K_3)=5$ and $I(C_3
\square C_3)=I(K_3 \square K_3)=2$ by Theorem \ref{theorem on KK}.
Also note that $Z(C_3 \square C_4)=Z(K_3 \square C_4)=6$ and
$I(C_3 \square C_4)=I(K_3 \square C_4) =
\lceil\frac{4-2}{2}\rceil=1$ by (iv) of Theorem \ref{T1}.
Further, one can check that $Z(C_4 \square C_4)=8$ and $I(C_4
\square C_4)=1$. Thus we have the following\\

\textbf{Conjecture.} Let $t > s \ge 3$. Then\begin{equation*} \left\{
\begin{array}{ll}
Z(C_s \square C_s) = 2s-1 & \mbox{ if $s$ is odd }\\
Z(C_s \square C_s) = 2s & \mbox{ if $s$ is even }\\
Z(C_s \square C_t) = 2s \ .
\end{array} \right.
\end{equation*}

%%%%%%%%%%%%%%%%%%%%%%%%%%%%%%%%
%%%%%%%%%%%%%%%%%%%%%%%%%%%%%%%%

\section{Upper bounds of iteration index of triangular grids and
and king grids}

The \textit{triangular grid graph}, denoted by $P_s \specsym P_t$, can be
obtained from the grid graph $P_s \square P_t$ by adding a
diagonal edge of negative slope to each $C_4$ square.
In \cite{B2}, it was shown that $Z(P_s
\specsym P_t)=s$ if $t \ge s \ge 2$. Figure \ref{P4-10(ex)} shows
a zero forcing set of $P_4 \specsym P_{10}$ and its forcing chain,
where the set of black
vertices is a zero forcing set $Z_0$ and the number $m$ in each vertex indicates that the vertex is
in $Z_m$.\\

\begin{figure}[htbp]
\begin{center}
\scalebox{0.45}{\input{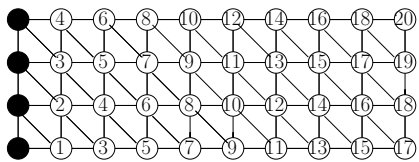}} \caption{The $4 \times
10$ triangular grid graph and its zero forcing
chain}\label{P4-10(ex)}
\end{center}
\end{figure}

\begin{Thm}
For $t \ge s \ge 2$, $I(P_s \specsym P_t) \leq 2t+s-4$.
\end{Thm}

\begin{proof}
Refer to (A) of Figure \ref{P4-10} for the labeling of vertices.
It's known that $Z(P_s \specsym P_t) =s$ for $2\leq s\leq t$. Take
$Z_0=\{(0,0),(0,1),\ldots,(0,s-1)\}$. We'll show that the vertex
$(i,j)\in Z_{2i+j-1}$ for $(i,j)\notin Z_0$: the theorem would
then follow since the range of the function $n=n(i,j)=2i+j-1$ over
the lattice $\Lambda=\{1,2,\ldots,t-1\}\times\{0,1,\ldots, s-1\}$
for $2\leq s\leq t$ is the set $\{1,2,\ldots,2t+s-4\}$. We'll
induct on $n\in \{1,2,\ldots,2t+s-4\}$.

\vspace{.1in}

We prove by strong induction. Let $n=1$. The only solution to $2i+j-1=1$ for
$(i,j)\in\Lambda$ is $(1,0)$. One sees immediately that $(1,0) \in Z_1$, since it's forced by $(0,0)\in Z_0$.

\vspace{.1in}

Suppose $(i,j)\in Z_{n=2i+j-1}$ for all
$(i,j)\in\Lambda$ such that $1\leq 2i+j-1<n_0$, where $2\leq
n_0\leq 2t+s-4$. We need to show, for $(i,j)$ with $2i+j-1=n_0$,
that $(i,j)\in Z_{n_0}$. Now, $(i,j)\in Z_{n_0}$ (``white vertex"
$(i,j)$ is turned ``black" in or before the $n_0$-th
iteration) if $(i,j)$ has a neighbor (``black vertex") $(i',j')$
such that $(x,y)\in N[(i',j')] \setminus \{(i,j)\}$ implies $2x+y-1 \leq n_{0}-1$
(i.e., the vertex $(x,y)$ has been turned ``black" in or before
the $(n_{0}-1)$-th iteration). We claim that the vertex $(i',j')$ may be
taken to be $(i-1,j)$: (B) of Figure \ref{P4-10} shows the local
picture where $|N[(i-1,j)]|=7$, the maximum possible; it's
trivially checked that $2x+y-1<n_0$ for any $(x,y)\in
N[(i-1,j)] \setminus \{(i,j)\}$.
\end{proof}

\begin{figure}[htbp]
\begin{center}
\scalebox{0.43}{\input{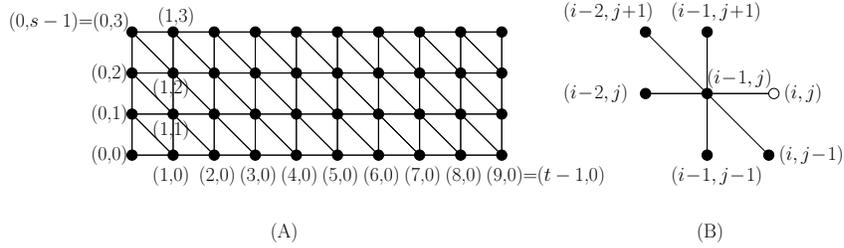}} \caption{The labeling of $4
\times 10$ triangular grid graph and $N[(i-1,j)]$}\label{P4-10}
\end{center}
\end{figure}

The \textit{king grid graph}, denoted by $P_s \boxtimes P_t$,  can be
obtained from the grid graph $P_s \square P_t$ by adding both
diagonal edges to each $C_4$ square. In \cite{AIM}, it was shown that $Z(P_s \boxtimes P_t)=s+t-1$ for $s, t \ge 2$. Figure \ref{king} shows $P_4 
\boxtimes P_{10}$ and
$P_3 \boxtimes P_{10}$, along with a zero forcing set and its forcing chain for each graph.
\begin{figure}[htbp]
\begin{center}
\scalebox{0.45}{\input{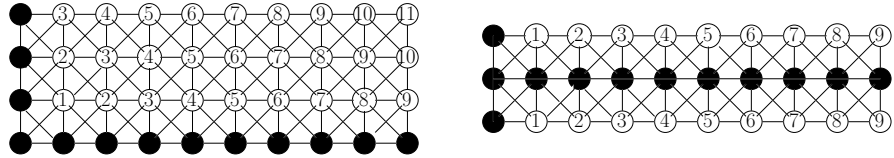}} \caption{$P_4 \boxtimes P_{10}$ and $P_3
\boxtimes P_{10}$, together with a zero forcing set for each graph: the number  $m$ in
each vertex indicates that the vertex is
in $Z_m$.}\label{king}
\end{center}
\end{figure}

\begin{Thm}
For $s,t \ge 2$, $I(P_s \boxtimes P_t) \le s+t-3$. In fact,
$I(P_3 \boxtimes P_t) \le t-1$ for $t \ge 2$.
\end{Thm}

\begin{proof}
Refer to Figure \ref{P4x10} for the labeling of vertices. Noting
$Z(P_s \boxtimes P_t)=s+t-1$, let $Z_0=(\cup_{j=0}^{s-1}\{(0,j)\})
\cup (\cup_{i=1}^{t-1}\{(i,0)\})$. We'll show that the vertex $(i,j) \in Z_{n=i+j-1}$ for $(i,j) \not\in Z_0$: the first assertion of the theorem would
then follow
since the range of the function $n=n(i,j)=i+j-1$ over the lattice $\Lambda=\{1,2,\ldots,t-1\}\times\{1,\ldots, s-1\}$
for $s, t \geq 2$ is the set $\{1,2,\ldots, s+t-3\}$. We'll induct on $n\in \{1,2,\ldots, s+t-3\}$. \\

We prove by strong induction. Let $n=1$. The only solution to $i+j-1=1$ for $(i,j) \in \Lambda$ is $(1,1)$. One sees immediately that $(1,1) \in Z_1$, since it's forced by 
$(0,0)\in Z_0$.\\

Suppose $(i,j)\in Z_{n_{0}-1}$ for all $(i,j)\in \Lambda$ such that $1\leq i+j-1 \leq n_{0}-1$, where $2\leq n_0 \leq 
s+t-3$.
We need to show, for $(i,j)$ with $i+j-1=n_0$, that $(i,j)\in Z_{n_0}$. Notice $(i-1, j-1)$ is adjacent to $(i,j)$, and it suffices to show that  
$(x,y)\in N[(i-1,j-1)] \setminus \{(i,j)\}$ implies $x+y-1<n_0$.
But this is obvious --- in view of the coordinates assigned to the vertices.
\vspace{.1in}

Next, consider the particular case of $P_3 \boxtimes P_t$ for $t \ge 2$. If we take
$Z_0=(\cup_{i=0}^{t-1}\{(i,1)\}) \cup \{(0,0),(0,2)\}$ with $|Z_0|=t+2$, then, for each $0 \le i \le t-2$, we have that $(i, 0) \rightarrow (i+1, 0)$ and 
$(i, 2) \rightarrow (i+1, 2)$.
Hence $Z_{t-1}=V(P_3 \boxtimes P_t)$; i.e., $I(P_3 \boxtimes P_t) \le t-1$.
\end{proof}

\begin{figure}[htbp]
\begin{center}
\scalebox{0.45}{\input{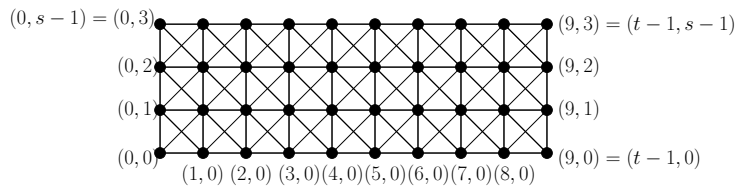}} \caption{The king grid
graph $P_4 \boxtimes P_{10}$}\label{P4x10}
\end{center}
\end{figure}

%%%%%%%%%%%%%%%%%%%%%%%%%%%%%%%%%%
%%%%%%%%%%%%%%%%%%%%%%%%%%%%%%%%%%

\section{Zero forcing number and iteration index of a bouquet of circles}

The bouquet of circles -- the figure $8$, in particular -- has been studied as a motivating example to introduce the fundamental group on a graph (see 
p.189, \cite{massey}). More recently, Llibre and Todd \cite{bouquet}, for instance, studied a class of maps on a bouquet of circles from a dynamical 
system perspective.\\

For $2 \le k_1 \le k_2 \le \ldots \le k_n$, let $B_n=(k_1, k_2,
\ldots, k_n)$ be a bouquet of $n \ge 2$ circles $C^1$, $C^2$,
$\ldots$, $C^n$, with the cut-vertex $v$, where $k_i$ is the
number of vertices of $C^i- \{v\}$ ($1 \le i \le n$). (The $n=1$
case has already been addressed.) Let $V(C^i)=\{v, w_{i,1}, w_{i,2},
\ldots, w_{i, k_i}\}$ such that $vw_{i,1} \in E(B_n)$ and $vw_{i,
k_i} \in E(B_n)$, and let the vertices in $C^i$ be cyclically
labeled, where $1 \le i \le n$. See Figure \ref{bouquetF} for
$B_3=(2,3,4)$. Note that $|V(B_n)|=1+ \sum_{i=1}^{n}k_i$, where
$k_i \ge 2$.\\

\begin{figure}[htbp]
\begin{center}
\scalebox{0.45}{\input{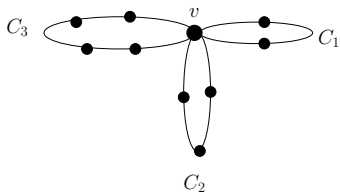}} \caption{A bouquet of three
circles, $B_3=(2,3,4)$}\label{bouquetF}
\end{center}
\end{figure}

\begin{Thm}\label{bouquetZ}
Let $B_n=(k_1, k_2, \ldots, k_n)$ be a bouquet of $n$ circles with
cut-vertex $v$. Then $Z(B_n)=n+1$.
\end{Thm}

\begin{proof}
One can readily check that $\{v\} \cup \{w_{i,1} \ | \ 1 \le i \le
n\}$ form a zero forcing set for $B_n$, and thus $Z(B_n) \le n+1$.
To prove the theorem, we need to show $Z(B_n) \ge n+1$. We make
the following claims.

\vspace{.11in}

\emph{Claim 1.} At least one vertex from each $C^i-\{v\}$ ($1
\le i \le n$) belongs to a $Z(B_n)$-set.

\vspace{.08in}

\textit{Proof of Claim 1:} This is clearly true by the assumption
that $k_i\geq 2$ for each $i$ -- one black vertex (namely $v$) on
$C^i$ can not ``force".

\vspace{.11in}

\emph{Claim 2.} Any $Z(B_n)$-set contains a pair of adjacent
vertices on a $C^i$ for some $i$.

\vspace{.08in}

\textit{Proof of Claim 2:} This is because the degree of every
vertex is at least two, and a set of isolated black vertices can
not force.

\vspace{.11in}

By Claims 1 and 2, we have $Z(B_n) \ge n+1$.
\end{proof}

\begin{Thm}
Let $B_n=(k_1, k_2, \ldots, k_n)$ be a bouquet of $n$ circles with
the cut-vertex $v$, where $n \ge 2$ and $k_i \ge 2$ ($1 \le i \le
n$). Then $I(B_n) = \left \lceil\frac{k_n+k_{n-1}}{2}\right\rceil-1$.
\end{Thm}

\begin{proof}
Let $Z_0$ be a $Z(B_n)$-set. First, assume that $Z_0$ contains the cut vertex $v$. We show $I_{Z_0}(B_n) = \left 
\lceil\frac{k_n+k_{n-1}}{2}\right\rceil-1$.
Of course, $I(B_n)\leq I_{Z_0}(B_n)$ by definition.\\

Notice that there is a unique $Z(B_n)$-set containing the
cut-vertex $v$ up to isomorphism of graphs. We take $Z_0=\{v\} \cup \{w_{i,1} \ |
\ 1 \le i \le n\}$. The presence of $v$ in $Z_0$ ensures that the
entire bouquet will be turned black as soon as the vertices in the longest cycle
$C^n$ are turned black. If $k_n \le k_{n-1}+1$, then the white
vertices of $C^n$ are turned black one at a time, and thus
$I_{Z_0}(B_n)=k_n-1=\lceil\frac{k_n+k_{n-1}}{2}\rceil-1$. If $k_n
\ge k_{n-1}+2$, then the white vertices of $C^n$ are turned black
one at a time until $(k_{n-1}-1)$-th step and two at a time
thereafter. This means that $V(B_n) \setminus \{w_{n,(k_{n-1})+1}, w_{n,
(k_{n-1})+2}, \ldots, w_{n, k_n}\}$ belongs to $Z_{(k_{n-1})-1}$,
and $1 \le |Z_{x+1}-Z_x| \leq 2$ for $x \ge k_{n-1}$. ($|Z_{x+1}-Z_x|$ may be less than 2 only if $x+1=I_{Z_0}(B_n)$.) Thus,
$I_{Z_0}(B_n)=k_{n-1}-1+\lceil\frac{k_n-k_{n-1}}{2}\rceil=\lceil\frac{k_n+k_{n-1}}{2}\rceil-1$. \\

Second, we show that if $Z_0$ does not contain $v$, then  $I_{Z_0}(B_n) \geq \left \lceil\frac{k_n+k_{n-1}}{2}\right\rceil-1$. \\

By Claims 1 and 2 in the proof of Theorem \ref{bouquetZ}, we have $2 \le |Z_0 \cap V(C^{n-1} \cup C^n)| \le 3$.
We note that the entire bouquet will not be turned black until all vertices in $C^{n-1} \cup C^n$ are turned black.
If $|Z_0 \cap V(C^{n-1} \cup C^n)|=2$, then $|Z_0 \cap V(C^{n-1})|=|Z_0 \cap V(C^n)|=1$, and forcing on $C^{n-1} \cup C^{n}$ can not start until $v$ 
is turned black. Since $v \in Z_m$ for some $m \ge 1$, by the same argument as in the upper bound case, we have
$I_{Z_0}(B_n) \ge m+\lceil\frac{k_n+k_{n-1}}{2} \rceil-1 \ge \lceil\frac{k_n+k_{n-1}}{2} \rceil$.
We can thus assume that $|Z_0 \cap V(C^{n-1} \cup C^n)|=3$. Observe that the lower bound is proved if we show that at most two vertices in $<V(C^{n-1} 
\cup C^n)>$ are turned black at a time. \\

WLOG, assume $|Z_0\cap C^{n-1}|=2$, as the other case where $|Z_0\cap C^{n}|=2$ is very similar. Then two vertices in $C^{n-1}$ are turned black at each 
step and no forcing occurs in $C^n$ until the cut-vertex $v$ is turned black. Now, if $v$ is the last vertex on $C^{n-1}$ to be turned black, then it's 
clear that at most two vertices in $<V(C^{n-1} \cup C^n)>$ are turned black at a time -- since, obviously, with any $Z(B_n)$-set at most two vertices are 
turned black on any cycle $C^i$ at each step. On the other hand, if $v$ is turned black before $C^{n-1}$ is turned entirely black, then $v$, being a 
neighbor to at least two white vertices, can not force until $C^{n-1}$ is turned entirely black. We've thus shown that $Z_m \cap V(C^{n-1} \cup C^n)$ 
contains at most two forcing vertices for any $m$ or, equivalently, at most two vertices in $<V(C^{n-1} \cup C^n)>$ are turned black at a time.
\end{proof}

\emph{Acknowledgement.} The authors wish to thank the anonymous referee for some suggestions and corrections.

\end{document}